\documentclass[11pt,reqno]{amsart}
\usepackage{tikz,adjustbox,lmodern}
\usepackage{cancel}
\usepackage{tikz-cd,amsmath,enumerate}
\usepackage{xcolor}
\usepackage{hyperref}
\hypersetup{
	colorlinks,
	citecolor=blue,
	linkcolor=black,
	urlcolor=blue}

\newcommand{\Z}{\mathbb{Z}}
\newcommand{\ux}{\underline{x}}
\newcommand{\uy}{\underline{y}}
\newcommand{\uz}{\underline{z}}

\newcommand{\mN}{\mathbb{N}}

\newcommand{\mR}{\mathbb{R}}

\newcommand{\mC}{\mathbb{C}}

\newcommand{\mV}{\mathbb{V}}

\newcommand{\mS}{\mathbb{S}}
\usepackage{enumerate}

\usepackage{hyperref}

\newcommand{\mcH}{\mathcal{H}}
\newcommand{\mcE}{\mathcal{E}}

\newcommand{\mcP}{\mathcal{P}}

\newcommand{\uv}{\underline{v}}
\newcommand{\uu}{\underline{u}}
\newcommand{\uw}{\underline{w}}

\newcommand{\sym}{\mathfrak{sp}}

\newcommand{\spl}{\mathfrak{sl}}
\newcommand{\so}{\mathfrak{so}}

\newcommand{\up}{\underline{\partial}}
\newcommand{\mE}{\mathbb{E}}

\usepackage[framemethod=tikz]{mdframed}
  
\newcommand{\U}{\operatorname{\mathsf{U}}} 
\newcommand{\R}{\mathbb{R}}

\usepackage[top=100pt,bottom=60pt,left=96pt,right=92pt]{geometry}

\newcommand{\Spin}{\operatorname{\mathsf{Spin}}}     
\newcommand{\C}{\mathbb{C}}

 \usepackage{amssymb}
 \newtheoremstyle{mystyle}
{}
{}
{\normalfont}
{0pt}
{\bfseries}
{.}
{4pt}
{}
\theoremstyle{mystyle}
\newtheorem{definition}{Definition}[section]

\newtheorem{notation}[definition]{Notation}
\newtheorem{remark}[definition]{Remark}

\newtheoremstyle{mystyle2}
{6pt}
{6pt}
{\itshape}
{0pt}
{\bfseries}
{.}
{4pt}
{}
\theoremstyle{mystyle2}
\newtheorem{theorem}[definition]{Theorem}
\newtheorem{prop}[definition]{Proposition}

\newtheorem{lemma}[definition]{Lemma}
\newtheorem{corollary}[definition]{Corollary}

\numberwithin{equation}{section}

\begin{document}

\title[The orthogonal branching problem for symplectic monogenics]{The orthogonal branching problem\\ for symplectic monogenics}


\author{David Eelbode}
\address{Campus Middelheim 
	\\Middelheimlaan 1
	\\M.G.312 
	\\2020 Antwerpen 
\\	Belgi\"e}
\curraddr{}
\email{david.eelbode@uantwerpen.be}
\thanks{}

\author{Guner Muarem}
\address{Campus Middelheim 
	\\Middelheimlaan 1
	\\M.G.221 
	\\2020 Antwerpen 
	\\	Belgi\"e}
\curraddr{}
\email{guner.muarem@uantwerpen.be}
\thanks{}


\date{October 2021}

\dedicatory{}

\begin{abstract}
In this paper we study the $\mathfrak{sp}(2m)$-invariant Dirac operator $D_s$ which acts on symplectic spinors, from an orthogonal point of view. By this we mean that we will focus on the subalgebra $\mathfrak{so}(m) \subset \mathfrak{sp}(2m)$, as this will allow us to derive branching rules for the space of $1$-homogeneous polynomial solutions for the operator $D_s$ (hence generalising the classical Fischer decomposition in harmonic analysis for a vector variable in $\mathbb{R}^m$). To arrive at this result we use techniques from representation theory, including the transvector algebra $\mathcal{Z}(\mathfrak{sp}(4),\mathfrak{so}(4))$ and tensor products of Verma modules.
\end{abstract}

\maketitle
 \bibliographystyle{math} 
\bibliography{untitled}
\setcounter{tocdepth}{1}
\tableofcontents
\section{Introduction}
\noindent
An important tool in understanding the representation theory for a Lie algebra $\mathfrak{g}$ is the study of restrictions of an irreducible representation $\rho$ for $\mathfrak{g}$ to a Lie subalgebra $\mathfrak{g}' \subset \mathfrak{g}$. This so-called restricted representation will in general no longer be irreducible as it decomposes into a direct sum of irreducible representations for $\mathfrak{g}'$. The algebraic tool which allows us to describe which $\mathfrak{g}'$-irreducible summands show up and what their multiplicities are is called a branching rule; see e.g. Goodman's and Wallach's book \cite{Goodman} for classical branching rules. These abstract rules get a physical meaning in the theory of symmetry breaking, see for instance \cite{Frahm} and the references therein for more details on symmetry breaking and its applications in physics (e.g. splitting of degenerate energy levels into multiplets in quantum mechanics). Moreover, the way in which $\mathfrak{g}'$-irreducible representations appear in $\mathfrak{g}$ often gives rise to special functions (the Gegenbauer polynomials appearing in harmonic analysis are an important example). \\
\noindent
The classical Fischer decomposition in harmonic analysis can be seen as the decomposition of the $\mathfrak{sp}(2m)$-irreducible Segal-Shale-Weil oscillator representation into irreducible representations for the algebra $\mathfrak{so}(m)$. \newpage \noindent In this paper we will generalise this result, focusing on the branching problem for other infinite-dimensional irreducible representations for the symplectic Lie algebra $\mathfrak{sp}(2m)$, when restricting to the (orthogonal) subalgebra $\mathfrak{so}(m)$. To do so, we will rely on the fact that the spaces of homogeneous solutions for the symplectic Dirac operator $D_s$ (known as symplectic monogenics) provide models for irreducible $\mathfrak{sp}(2m)$-representations of infinite dimension (these will be denoted by means of $\mS^\infty_k$ with $k \in \mN$).\\
The Dirac operator appears in many branches in mathematics, most notably in differential geometry and representation theory, but the study of its null solutions has always been an important objective in Clifford analysis. Traditionally, this operator is defined as the $\mathsf{Spin}(m)$-invariant first order operator factorising the Laplace operator, see for instance \cite{BDS, DSS, GM}. However, in recent years many variants of the classical theory were developed, in particular a symplectic version of the Dirac operator. This operator was introduced in \cite{Habermann} for general symplectic manifolds $(M,\omega)$ admitting a metaplectic structure (this is the symplectic analogue of the spin structure). Whereas the classical Dirac operator acts on differentiable functions with values in a {\em finite-dimensional} spinor space, often denoted by $\mathbb{S}$ and containing complex Weyl or Dirac spinors depending on the parity of the underlying dimension, the symplectic Dirac operator acts on differentiable functions with values in an {\em infinite-dimensional} symplectic spinor space $\mathbb{S}^{\infty}_0$ (the meaning of the subscript zero will become clear in what follows). These symplectic spinors were introduced by Kostant in \cite{Kos} and studied by Reuter \cite{Reuter} from a more physical point of view. In the work of Habermann, these symplectic spinors were modelled as elements of the representation space induced by the unitary representation $\mathfrak{m}:\mathsf{Mp}(m)\to \U(L^2(\R^m))$. The smooth vectors of $\mathfrak{m}$ then exactly correspond to the elements of the Schwartz space $\mathcal{S}(\R^m)$ of rapidly decreasing functions on ${\R}^{m}$. By using the multi-index notation, one introduces the following norm (for $\alpha ,\beta \in \mathbb {N} ^{m}$):  
$$||f||_{\alpha ,\beta }:=\sup _{\underline{q}\in \mathbb {R} ^{m}}\left|\underline{q}^{\alpha }(D^{\beta }f)(\underline{q})\right| .$$
The Schwartz space is then explicitly given by
$$
\mathcal{S}\left(\mathbb{R}^{m},\mathbb{C} \right):=\left\{f\in C^{\infty }(\mathbb {R}^{m},\mathbb {C} ) : \|f\|_{\alpha ,\beta }<\infty\  \text{for all } \alpha ,\beta \in \mathbb {N} ^{m} \right\}.$$
Using this realisation for $\mathbb{S}^{\infty}$, Habermann defined the symplectic Dirac operator on the canonical symplectic space $(\mathbb{R}^{2m},\omega_0)$, equipped with coordinates $(x_j,y_j)_{j=1}^m \in \R^{2m}$ and the standard symplectic form $\omega_0=\sum_{j=1}^m dx_j\wedge dy_j$, as the first order differential operator 
\[ D_s = i\langle \underline{q},\up_y\rangle - \langle \up_x,\up_q\rangle\ . \]
This operator was then studied more extensively in the work of De Bie et.\ al, see {e.g. \cite{BHS}}. In this paper we continue the study of the symplectic Dirac operator $D_s$, but we will switch from the aforementioned Schwartz model for symplectic spinors to a polynomial model. These latter is also known as the Fock model, while the former is often referred to as the Schr\"odinger model in the context of representation theory, see for instance \cite{berg}.
The {reason for switching to the Fock model} lies in the fact that we can then use techniques for decomposing spaces of polynomials in a matrix variable (under the action of the orthogonal group) to study symplectic monogenics in an encompassing framework,
{although we will soon see that not all vector variables will play the same role.} In particular, we define the \textit{symplectic Dirac operator} as an operator in the Weyl algebra $\mathcal{W}(\R^{3m})$ in three vector variables in $\mR^m$, by means of 
\[{D}_s:=\langle \underline{z},\underline{\partial}_y\rangle - \langle \underline{\partial}_x,\underline{\partial}_z \rangle. \]
We are thus working with three vector variables in $\mR^m$, i.e. a matrix-variable $(\ux,\uy;\uz)$ in $\mathbb{R}^{m\times 3}$ where we use the semicolon to stress the fact that $\uz$ plays another role than $(\ux,\uy)$. Indeed, whereas the latter denote the variables on which our functions (often polynomials) depend, the former will be used to the describe the symplectic spinors. These can indeed be modelled by polynomials in a variable $\uz \in \mR^m$, see for instance the work of Krýsl \cite{Krysl}.\\
Note that our setup has implications for the dimensional parameter $m \in \mN$: in order to be sure that tensor products of $\so(m)$-representations behave `nicely', one has to work in the so-called \textit{stable range}; see the paper of Howe et al.\ \cite{HT} for more background. This means that twice the number of variables ($3$ in our case) has to be less than the dimension $m$. Put differently, throughout this paper we work with $m \geq 6$. It is of course a relevant question to study what happens outside this stable rangle, see for instance the recent paper \cite{Lav} for the implications on the classical Howe duality in harmonic analysis. The object of study in this paper will then be the following spaces: 
\begin{align*}
\mathbb{S}_k^{\infty}:=\ker_k(D_s) = \ker(D_s)\cap\left( \mathcal{P}_k(\mathbb{R}^{2m},\mathbb{C})\otimes\mathcal{P}(\R^m,\C)\right),
\end{align*}
where $\mathcal{P}(\R^m,\C)$ is the notation we use for $\mathbb{C}$-valued polynomials (with a subscrip $k$ to denote the subspace containing $k$-homogeneous polynomials). The infinite-dimensional symplectic module $\mathbb{S}_k^{\infty}$ is irreducible, see for instance \cite{BHS}. The branching problem we will consider is thus denoted as follows: which highest $\so(m)$-weights $\lambda$ appear in the direct sum decomposition 
\begin{align}\label{branchings}
\mathbb{S}^{\infty}_k\bigg\downarrow^{\mathfrak{sp}(2m)}_{\mathfrak{so}(m)} = \bigoplus_{\lambda}\mu_\lambda\mV_\lambda,
\end{align}
with their multiplicities $\mu_\lambda$. Right from the onset, we should note that this is in a sense ill-defined: $\mu_\lambda$ will {\em not} be a number, but an infinity. It will thus be crucial to keep track of these infinities using a dual symmetry algebra, this is where Verma modules will appear. Not only will we determine which $\mathfrak{so}(m)$-irreducible representations will appear in the symplectic module $\mathbb{S}_k^{\infty}$, we will also investigate how these modules are embedded (whereby the case $k=1$ will be our guiding example). These embeddings are provided by a set of rotationally invariant operators which appeared in the literature as the generators of the \textit{transvector algebra} $\mathcal{Z}(\mathfrak{sp}(4),\so(4))$. This quadratic algebra was used by De Bie et al. in \cite{BER} as a new dual partner for the orthogonal group, hence obtaining formulas for the integral over a Stiefel manifold. See also the paper of Zhelobenko \cite{Zhelobenko} for a general introduction to the theory of transvector algebras, and section 4 for the necessary background in this paper.
Finally, we mention that the `opposite question' has been treated before: one can start from orthogonal spinors (irreducible for the orthogonal Lie algebra in a suitable dimension) and decompose this space into irreducible representations for the symplectic Lie algebra. As a matter of fact, this question gave rise to so-called quaternionic hermitian Clifford analysis and the study of $\mathfrak{osp}(4|2)$-monogenics, see for instance \cite{BSE}. From this point of view, the present paper is the inverse problem. It is more subtle though, because of the ininities mentioned above. \\
Let us then briefly sketch the outline of this paper: in Section 2, the classical Fischer decomposition will be recalled, as this allows to decompose the space of polynomials into harmonic polynomials (this then serves as our guiding example of an infinite-dimensional symplectic branching problem). In Section 3 we provide more information on symplectic Clifford analysis and introduce the symplectic Dirac operator in our polynomial model for the symplectic spinors. In Section 4, we derive our main tool to solve the branching problem: a new Fischer decomposition for the space $\mathcal{P}(\R^{3m},\mathbb{C})$. \newpage \noindent Next, in Section 5, we use this Fischer decomposition, the symplectic Howe duality and the transvector algebra to arrive at our desired branching rule. In the last section we will look at this branching problem from the dual point of view: although this does not lead to the full picture, it does at least convey some useful insight. 
\begin{notation}
In what follows we will need various realisations of the Lie algebra $\mathfrak{sl}(2)$. In order to fix notations for the rest of the paper, we recall that in matrix language this Lie algebra is defined by the following traceless $(2\times 2)$-matrices:
\begin{align*}
X = \begin{pmatrix}0&1\\0&0\end{pmatrix},\quad
Y = \begin{pmatrix}0&0\\1&0\end{pmatrix}\quad\text{and}\quad
H = \begin{pmatrix}1&0\\0&-1\end{pmatrix}.
\end{align*}
These basis elements satisfy the relations $[X,Y]=H$, $[H,X] = 2X$ and $[H,Y]=-2Y$. We note that all the Lie algebras encountered in this paper are assumed to be complex. The notation $\operatorname{Alg}(A_1,\dots,A_n)$ will always stand for the Lie algebra generated by the operators $A_1,\dots,A_n$ under the commutator bracket $[A_k,A_{\ell}]:=A_kA_{\ell}-A_{\ell}A_k$. We adopt the notation $\langle \uv,\uw\rangle:=\sum_{j=1}^m u_jv_j$ for the usual inner product on $\mathbb{R}^m$, and whenever $\mN$ appears we mean all the positive integers including zero.
\end{notation}
\subsection{Acknowledgments} 
The second author was supported by the {FWO-EoS project} \texttt{G0H4518N}.
\section{The classical Fischer decomposition}
\noindent
Consider the space $\mathcal{P}(\R^m,\C)$ of complex valued polynomials defined on $\R^m$. There is a canonical action of the group $\mathsf{SO}(m)$ on this space. For $g\in\mathsf{SO}(m)$ and $P(\uz)\in\mathcal{P}(\R^m,\C)$ we define $g\cdot P(\uz):= P(g^{-1} \uz)$ as the so-called \textit{regular action}.
\begin{definition}
Let $\Delta_z=\sum_{j=1}^m\partial_{z_j}^2$ be the Laplace operator on $\R^m$. We then say that a function $f(\uz)$ on $\R^m$ is \textit{harmonic} if $f(\uz) \in \ker(\Delta_z)$. We denote the space of \textit{k-homogeneous harmonics} by means of $\mathcal{H}_k(\R^m,\mathbb{C}):=\mathcal{P}_k(\R^m,\mathbb{C})\cap \ker(\Delta_z)$. We hereby stress that the variable $z_j \in \mR$ is real, it should not be read as a complex variable here (it corresponds to the $j$-th coordinate for the vector variable $\uz \in \mR^m$). 
\end{definition}
\begin{notation}
It is well-known that the space $\mathcal{H}_k(\R^m,\mathbb{C})$ provides a polynomial model for the irreducible representation for SO$(m)$, or its Lie algebra $\so(m)$, with highest weight $(k,0,\ldots,0)$. In this paper we will use the short-hand notation $(k)_z$ for this highest weight, so trailing zeroes will often be omitted for notational ease and a subscript will be added to make clear that this space is realised in terms of polynomials depending on $\uz \in \mR^m$ (or any other vector variable, when needed).
\end{notation}
\noindent
We can now formulate the {Fischer decomposition} of the polynomial space $\mathcal{P}(\R^m,\C)$ into the space of harmonic homogeneous polynomials $\mathcal{H}_k(\R^m,\C)$ of degree $k$ by means of 
\[ \mathcal{P}(\R^m,\C) = \bigoplus_{k=0}^{\infty} \bigoplus_{p=0}^{\infty}  |\uz|^{2p} \mathcal{H}_k(\R^m,\C). \]
All subspaces $|\uz|^{2p}\mathcal{H}_k(\R^m,\mathbb{C})$ for fixed $k$ are isomorphic as $\mathsf{SO}(m)$-modules (since $|\uz|^2$ is an invariant). This implies that every irreducible representation $\mathcal{H}_k(\R^m,\mathbb{C})$ appears infinitely many times in this decomposition. In other words, the decomposition from above is not multiplicity-free. However, by using the `hidden' $\mathfrak{sl}_h(2)$-symmetry of the polynomial space $\mathcal{P}(\R^m,\C)$, where $\mathfrak{sl}_h(2)$ is realised as
\[ \spl_h(2) = \textup{Alg}\bigg(\frac{1}{2}|\uz|^2, -\frac{1}{2}\Delta_z, \mE_z + \frac{m}{2}\bigg), \]  
we will be able to gather the infinitely many copies of $\mathcal{H}_k$ into one irreducible representation for the Lie algebra $\mathfrak{sl}_h(2)$. Consequentially, the decomposition becomes multiplicity-free under the joint action of the \textit{Howe dual pair} $(\mathsf{SO}(m),\mathfrak{sl}(2))$, see \cite{Goodman} for more details on dual pairs. The `gathering' mentioned above is done by introducing infinite-dimensional Verma modules for the Lie algebra $\mathfrak{sl}(2)$. As these will play a key role in this paper, we briefly recall their construction. 
\begin{definition}
	Let $\mathfrak{g}$ be a complex simple Lie algebra with Borel subalgebra $\mathfrak{b}$, Cartan subalgebra $\mathfrak{h}$, $\lambda\in\mathfrak{h}^*$ and denote by $\C_{\lambda}$ an irreducible representation of $\mathfrak{b}$. We now define the infinite-dimensional \textit{Verma module} with lowest weight $\lambda$ as the $\mathcal{U}(\mathfrak{g})$-module induced by $\C_{\lambda}$. In other words, we have
	\[ \mathbb{V}^{\infty}_{\lambda}:= \mathcal{U}(\mathfrak{g})\otimes_{\mathcal{U}(\mathfrak{b})} \C_{\lambda}.\]
\end{definition}
\noindent
In our case, $\mathfrak{g}=\mathfrak{sl}(2)$ and we obtain the following more explicit definition. Let $v$ be a \textit{lowest weight} vector, i.e. $Hv = \lambda v$ and $Yv= 0$ (where e.g. $Hv := H[v]$ denotes the action of $H$ on $v$).  
This means that the action of $H$ on $v$ is diagonal and $v$ is annihilated by $Y$. The Verma module $\mathbb{V}^{\infty}_{\lambda}$ is then constructed as the infinite dimensional module generated by the lowest weight vector $v$. Letting the operator $X$ act on $v$ we thus get
\[ \mathbb{V}_\lambda^\infty := \textup{span}_\mathbb{C}\big\{X^k v : k \in \mathbb{N}\big\}\ . \]
One easily verifies that the action of the $H$ and $Y$ on a basis vector is given by: 
\begin{align*}
	{H(X^nv)} &= (\lambda + 2n)(X^nv)\\
	{Y(X^nv)} &=-n({\lambda}+n-1)(X^{n-1}v)\ . 
\end{align*}
Using this notation for Verma modules, we can state the following (where we have switched from the orthogonal group to the corresponding Lie algebra): 
\begin{theorem}[Multiplicity-free Fischer decomposition, \cite{Goodman}]
	Under the joint action of  $\so(m) \times \mathfrak{sl}(2)$, we obtain the following \textit{multiplicity-free} decomposition of the polynomial space $\mathcal{P}(\R^m,\C)$: 
	\[ \mathcal{P}(\R^{m},\C) \cong \bigoplus_{k=0}^{\infty}\mathcal{H}_k(\R^m,\C)\otimes \mathbb{V}_{k+\frac{m}{2}}^{\infty},\]
	where $\mathbb{V}_{k+\frac{m}{2}}^{\infty}$ is the infinite-dimensional Verma module defined above, for the algebra $\mathfrak{sl}_h(2)$.
\end{theorem}
\begin{corollary}\label{mpfree}The multiplicity-free decomposition from above allows us to conclude:
\begin{enumerate}[\normalfont(i)]
\item  $\mathcal{H}_k(\R^m,\C)\otimes \mathbb{V}_{k+\frac{m}{2}}^{\infty}$ is an irreducible $\so(m) \times \mathfrak{sl}(2)$-module.
\item As an $\mathfrak{sl}(2)$-module, this tensor product contains $\dim(\mathcal{H}_k)$ copies of $\mathbb{V}_{k+\frac{m}{2}}^{\infty}$.
\item As an $\mathfrak{so}(m)$-module, it contains infinitely many copies of $\mathcal{H}_k$ (as the Verma module has infinite dimension). However, this Verma module allows us to keep track of this infinity. 
\end{enumerate}
\end{corollary}
\noindent
This Fischer decomposition can also be restated in such a way that the connection with the $\sym(2m)$-irreducible representations $\mS^\infty_{0\pm}$ becomes clear. Note that we added a parity sign $(\pm)$ to the symplectic spinors in the theorem below, as both in the Schwartz model (see the introduction) or the polynomial model one should make a distinction between even and odd spinors. However, as the impact of this parity sign on the analysis is minimal (it only means that when considering polynomials in $\uz \in \mR^m$, the degree has to be even or odd) we will ignore the parity sign throughout the paper.
\begin{theorem}\label{s0infty}
The branching problem $\mathbb{S}^{\infty}_0\big \downarrow^{\mathfrak{sp}(2m)}_{\mathfrak{so}(m)}$ corresponds to the harmonic Fischer decomposition of the space $\mathcal{P}(\R^m,\mathbb{C})$.
\end{theorem}
\begin{proof}
		We note that for $k=0$ we have $\mathbb{S}_0^{\infty}=\mathcal{P}(\R^m,\mathbb{C}).$ This space can be written as a direct sum of the spaces of even and odd polynomials. This naturally leads to the orthogonal branching for the space of symplectic spinors (the parity of the polynomials dictates the parity of the spinors). In terms of $\spl(2)$-Verma modules, one has that 
\begin{align*}
\mS^\infty_{0^+}  \cong  \bigoplus_{k = 0}^\infty \mcH_{2k} \otimes \mV^\infty_{2k + \frac{m}{2}}\quad\text{and}\quad 
\mS^\infty_{0^-} \cong  \bigoplus_{k = 0}^\infty \mcH_{2k+1} \otimes \mV^\infty_{2k + 1 + \frac{m}{2}}\ .
\end{align*}
All together this means that
\[ \mathbb{S}_0^{\infty}=\mathbb{S}_{0^+}^{\infty}\oplus \mathbb{S}_{0^-}^{\infty}\cong\bigoplus_{k=0}^{\infty} \mathcal{H}_{k}(\mR^m,\mC) \otimes \mathbb{V}^{\infty}_{\frac{m}{2}+k}\ ,\] 
so that we can see the $\mathfrak{so}(m)$-irreducible representations of the $\mathfrak{sp}(2m)$-module $\mathbb{S}^{\infty}_0$.
\end{proof}
\noindent
In the remainder of this paper we will first derive a similar result for $k=1$. We will give an explicit argument based on a novel Fischer decomposition of the space $\mathcal{P}_1(\R^{2m},\mathbb{S}_0^{\infty})$. Once the decomposition for $\mathbb{S}_1^{\infty}$ as a representation for $\so(m)$ is known, this can then be used inductively to solve the branching problem for $\mathbb{S}_k^{\infty}$ for general $k$ (this will then be explained at the end of the paper). However, to do so we will also need more general irreducibele representations for $\so(m)$ (than just harmonics). These $\so(m)$-modules have a highest weight of the form $(\lambda_1,\lambda_2,\ldots,\lambda_p,0,\ldots,0)$ where the integers satisfy the dominant weight condition $\lambda_1\geq \lambda_2\geq \dots \geq \lambda_p \geq 0$, and can be realised using harmonic polynomials in several vector variables. For the full details we refer to \cite{LPC}, we shortly review the basics which we will need here (note that also for these highest weights trailing zeroes will often be omitted). In the following, let $N$ be an integer for which $1\leq N\leq \lfloor m/2\rfloor$. \begin{definition}\label{simplicial}
	A function $f:\mathbb{R}^{N\times m}\to\mathbb{C}$ given by $(\uu_1,\dots,\uu_N) \mapsto f(\uu_1,\dots,\uu_N)$ is called \textit{simplicial harmonic} if the following holds:
	\begin{align*}
	\langle \up_{u_i},\up_{u_j}\rangle f = 0 \text{ for } (i,j=1,\dots,N)
	\quad	\text{and}\quad\langle \uu_i,\up_{u_j}\rangle f = 0	 \text{ for } (1\leq i < j\leq N).
	\end{align*}
	The vector space of all simplicial harmonic polynomials, which are $k_i$-homogeneous in the vector variable $\uu_i$ is then denoted by $\mathcal{H}_{k_1,\dots,k_N}(\mathbb{R}^{N m},\mathbb{C})$.
\end{definition}
\section{Symplectic Clifford analysis and the symplectic Fischer decomposition}
\noindent
In this section we review some of the standard facts on symplectic Clifford analysis which can be e.g. found in the work of 
De Bie \cite{BHS}, Habermann \cite{Habermann}, and Hol\'ikov\'a \cite{Hol}. Let $V$ be a symplectic vector space of even dimension equiped with symplectic form $\omega$. Then we define the symplectic Clifford algebra as follows:
\begin{definition}
	 The \textit{symplectic Clifford algebra} $\mathsf{Cl}(V,\omega)$ is defined as the quotient algebra of the tensor algebra $T(V)$ of $V$, by the two-sided ideal $\mathcal{I}_{\omega}$ generated by elements of the form $\{v\otimes u-u\otimes v+\omega(v,u) : u,v\in V\}$. In other words
	$\mathsf{Cl}(V,\omega):=T(V)/\mathcal{I}_{\omega}$
	is the algebra generated by $V$ in terms of the relation $v u-u v=-\omega(v,u)$, where we have omitted the tensor product symbols.
\end{definition}
\noindent
From now, we take $V=(\mathbb{R}^{2m},\omega_0)$ with the symplectic structure $\omega_0=\sum_j dx_j\wedge dy_j$.
\begin{remark}
We note that the (universal) Clifford algebra in the orthogonal setting is defined in a similar way, but using a {\em symmetric} bilinear form. For a comparison between orthogonal and symplectic Clifford algebras we refer the reader to Crumeyrolle's book \cite{Crum}. Note that the impact of this transition from a symmetric to an anti-symmetric form is rather dramatic: the symplectic Clifford algebra is infinite-dimensional, whereas in the orthogonal case it is of finite dimension. This has major consequences concerning the representation theory of the symplectic spinors, in the sense that infinite-dimensional Verma modules will naturally appear.
\end{remark}
\noindent
The symplectic Clifford algebra can also be seen as the Weyl algebra $\mathcal{W}(\mR^{2m})$, i.e. the universal enveloping of the Heisenberg algebra generated by $2m$ commuting variables and their associated partial derivatives, with for instance $[\partial_{x_i},x_j] =\delta_{ij}$. 

\begin{lemma}\label{sprealisation}
The symplectic Lie algebra $\mathfrak{sp}(2m)$ has the following realisation on the space of symplectic spinor-valued polynomials $\mathcal{P}(\mathbb{R}^{2m},\mathbb{C})\otimes\mathcal{P}(\mathbb{R}^m,\C)$:
 	\begin{align}\label{realisaties}
\begin{cases}
X_{jk}=x_j\partial_{x_k}-y_k\partial_{y_j} - (z_k\partial_{z_j}+\frac{1}{2}\delta_{jk})
&\quad j,k=1,\dots,m\quad\qquad m^2
\\Y_{jk}=x_j\partial_{y_k}+x_k\partial_{y_j} -\partial_{z_j}\partial_{z_k}
&\quad j<k=1,\dots,m\quad \frac{m(m-1)}{2}
\\ Z_{jk}=y_j\partial_{x_k}+y_k\partial_{x_j} +  z_jz_k
&\quad j<k=1,\dots,m\quad \frac{m(m-1)}{2}
\\Y_{jj}=x_j \partial_{y_j} -\frac{1}{2}\partial_{z_j}^2
&\quad j=1,\dots,m\qquad\qquad m
\\Z_{jj}=y_j\partial_{x_j}+\frac{1}{2} z_j^2
&\quad j=1,\dots,m\qquad\qquad m
\end{cases}
\end{align} 
\end{lemma}
\noindent
The power of the realisation above (which can be verified by straightforward calculations) lies in the fact that everything can be formulated on the level of operators acting on for instance polynomials in the matrix variable $(\ux,\uy,\uz) \in \mR^{m \times 3}$. Also note that the operators in $\uz$ are homogeneous of degree 2, which explains why the parity for symplectic spinors is related to the degree in $\uz \in \mR^m$. 
\begin{definition}
	The following operators, acting on $\mathcal{P}(\mR^{m \times 3},\mC)$, are $\mathfrak{sp}(2m)$-invariant:
	\begin{enumerate}[(i)]
		\item The \textit{symplectic Dirac operator}, given by ${D}_s=\langle \underline{z},\underline{\partial}_y\rangle - \langle \underline{\partial}_x,\underline{\partial}_z \rangle$.  
		\item The \textit{dual symplectic Dirac operator} with respect to the symplectic Fischer inner product (see \cite{BHS} for more details), given by $D_s^{\dagger}=\langle \underline{y},\underline{\partial}_z\rangle + \langle \underline{x},\underline{z}\rangle$. 
		\item The \textit{Euler operator} which measures the degree of homogeneity $k$ in the variables $(\ux,\uy)$, given by $\mathbb{E}=\mathbb{E}_x + \mathbb{E}_y$. 
	\end{enumerate}
The invariance of these operators is easily verified, as each of these operators commutes with the generators for $\mathfrak{sp}(2m)$ from the lemma above. These operators also generate a Lie algebra under the commutation bracket: it is isomorphic to $\mathfrak{sl}(2)$, and will be denoted by $\mathfrak{sl}_c(2)$ in this paper (we use the letter `c' the refer to the double `c' in the word \textit{symplectic}).
\end{definition}
\begin{remark}
	In the orthogonal case, the Dirac operator $\up_x$ has a (Fischer) dual given by the vector variable $\ux=\sum_{j = 1}^me_jx_j$ (interpreted as a multiplication operator, acting on spinor-valued functions). In a certain sense, the operator $D_s^{\dagger}$ is the symplectic equivalent of this `vector variable'. However, note that $D_s^{\dagger}$  contains derivatives with respect to the (spinor) variable $\uz$. We further note that $D_s$ and $D_s^{\dagger}$ are \textit{not} homogeneous in the $\uz$-variable.
\end{remark}
\noindent
In \cite{BSS} De Bie, Somberg and Sou\v{c}ek proved a symplectic Fischer decomposition for the symplectic Dirac operator. Their realisation was slightly different, as they worked with the Schwartz space to model symplectic spinors, but this result remains valid in our realisation. It suffices to use the identification $iq_j \mapsto z_j$ and $\partial_{q_j} \mapsto \partial_{z_j}$ for the Dirac operator and to map the Gaussian (i.e. the lowest weight vector for the symplectic spinor space) to the (constant) polynomial $1 \in \mcP(\mR^m,\mC)$. 

\begin{theorem}[Modified symplectic Fischer decomposition]\label{decomp}
Under the action of $\mathfrak{sp}(2m)$ the space of polynomials with values in the space of symplectic spinors decomposes as follows:
\begin{align}\label{decomposition}
\mathcal{P}(\mathbb{R}^{2m},\C)\otimes\mathcal{P}(\mathbb{R}^m,\C) =\bigoplus_{k=0}^{\infty}\bigoplus_{j=0}^{\infty} {(D_s^{\dagger}})^j\mathbb{S}_k^{\infty},
\end{align}
where $\mathbb{S}_k^{\infty}:=\ker_k(D_s)= \ker(D_s) \cap (\mathcal{P}_k(\R^{2m},\C) \otimes \mathcal{P}(\R^m,\C))$.
This gives rise to the following infinite-dimensional triangular structure, with Howe dual pair $\sym(2m) \times \mathfrak{sl}_c(2)$: 
\[
\begin{tikzcd}[row sep=-0.2em,column sep=1em]
\mathcal{P}_0\otimes \mathbb{S} & \mathcal{P}_1\otimes \mathbb{S} & \mathcal{P}_2\otimes \mathbb{S} & \mathcal{P}_3\otimes \mathbb{S} & \mathcal{P}_4\otimes \mathbb{S} & \mathcal{P}_5\otimes \mathbb{S} & \cdots
	\\ \vert\vert & \vert\vert & \vert\vert & \vert\vert & \vert\vert & \vert\vert
	\\ \mathbb{S}_0^{\infty}  & {D^{\dagger}}\mathbb{S}_0^{\infty}  & {D^{\dagger}}^2\mathbb{S}_0^{\infty}  & {D^{\dagger}}^3\mathbb{S}_0^{\infty}  & {D^{\dagger}}^4\mathbb{S}_0^{\infty}  & {D^{\dagger}}^5\mathbb{S}_0^{\infty}
	\\ & \oplus & \oplus & \oplus & \oplus & \oplus
	\\ & \mathbb{S}_1^{\infty}  & {D^{\dagger}}\mathbb{S}_1^{\infty}  & {D^{\dagger}}^2\mathbb{S}_1^{\infty}  & {D^{\dagger}}^3\mathbb{S}_1^{\infty}  & {D^{\dagger}}^4\mathbb{S}_1^{\infty} 
	\\ &  & \oplus & \oplus & \oplus & \oplus
	\\ & & \mathbb{S}_2^{\infty}  & {D^{\dagger}}\mathbb{S}_2^{\infty}  & {D^{\dagger}}^2\mathbb{S}_2^{\infty}  & {D^{\dagger}}^3\mathbb{S}_2^{\infty}
	\\ &  &  & \oplus & \oplus & \oplus
	\\ & & & \mathbb{S}_3^{\infty}  & {D^{\dagger}}\mathbb{S}_3^{\infty}  & {D^{\dagger}}^2\mathbb{S}_3^{\infty}
	\\ & & &  & \oplus & \oplus 
	\\ & & & & \mathbb{S}_4^{\infty}  & {D^{\dagger}}\mathbb{S}_4^{\infty}
	\\ & & & & & \oplus 
	\\ & & & & & \mathbb{S}_5^{\infty}
	\end{tikzcd}
	\]
\end{theorem}
\noindent
In this paper, we are thus interested in the decomposition of the spaces appearing at the left-most edge of this triangle.  

 
\section{A new Fischer decomposition for symplectic spinor-valued polynomials}
\noindent
As we observed several times already, various realisations of the Lie algebra $\mathfrak{sl}(2)$ can be put to use to decompose a space of polynomials into irreducible components for a dual symmetry algebra. First of all, we arrived at the classical Fischer decomposition of the space $\mathcal{P}(\R^m,\C)$ into irreducible $\mathfrak{so}(m)$-modules using the (harmonic) Lie algebra $\spl_h(2)$. Secondly, the space $\mathcal{P}(\R^{2m},\mathbb{S}_0^{\infty})$ of symplectic spinor-valued polynomials was decomposed in $\mathfrak{sp}(2m)$-irreducible representations using the algebra $\mathfrak{sl}_c(2)$.
The goal of this section is to decompose the space \[ \mathcal{P}_1(\mathbb{R}^{2m},\mathbb{C})\otimes \mathbb{S}^{\infty}_0 \cong \mathcal{P}_1(\mathbb{R}^{2m}) \otimes \mathcal{P}(\R^m,\mathbb{C}) \]
using a suitable copy of the Lie algebra $\spl(2)$, for which we will again use a special letter as a subscript. First of all, apart from the algebra $\mathfrak{sl}_h(2)$ introduced earlier (acting on the variable $\uz \in \mathbb{R}^m$ used to represent symplectic spinors) we will also need the following realisation:
\begin{align}\label{sl2}
\spl_s(2) &:= \textup{Alg}\big(\langle \uy,\up_x \rangle, \langle \ux,\up_y \rangle, \mE_y - \mE_x\big),
\end{align}
where we used the subscript `$s$' referring to `skew', a common way to denote these operators in the context of Clifford analysis in several vector variables. Note that $\spl_c(2)$ acts solely on the variables $(\ux,\uy) \in \mathbb{R}^{2m}$, whereas $\spl_h(2)$ deals with the spinor space in the variable $\uz \in \mathbb{R}^m$. \newpage \noindent The direct sum of both (commuting) algebras $\spl_s(2)$ and $\spl_h(2)$ yields a realisation of the orthogonal Lie algebra $\so(4)$, and inside the latter one can then look at the so-called diagonal embedding $\spl_d(2) \hookrightarrow \spl_s(2) \oplus \spl_h(2)$. This will be our suitable realisation mentioned earlier this section. 
\begin{lemma} The operators 
	\begin{align*}
	L &:= \sum_{j = 1}^m Y_{jj} = \langle \ux,\up_y \rangle - \frac{1}{2}\Delta_z\\ 
	R &:= \sum_{j = 1}^m Z_{jj} = \langle \uy,\up_x \rangle + \frac{1}{2}|\uz|^2 \\ 
 \mathcal{E} &:= \mE_y - \mE_x + \mE_z + \frac{m}{2}.
	\end{align*}
generate an algebra isomorphic to $\mathfrak{sl}(2)$. We denote this algebra by $\mathfrak{sl}_{d}(2)$. 
\end{lemma}
Note that the operators $Y_{jj}$ and $Z_{jj}$ appeared in Lemma \ref{sprealisation} as generators for the Lie algebra $\mathfrak{sp}(2m)$. 
\begin{proof}
Follows from straightforward commutators.
\end{proof}
\begin{remark}
We have chosen the letters $R$ and $L$ here, for raising and lowering. The role played by these ladder operators will become clear in what follows. Moreover, for polynomials in the variable $\uz \in \mR^m$ only (i.e. symplectic spinors in $\mS^\infty_0$) the Lie algebra $\mathfrak{sl}_d(2)$ can be recognised as $\mathfrak{sl}_h(2)$. This already indicates that the algebra $\spl_d(2)$ generalises the role played by $\spl_h(2)$ in the classical (harmonic) Fischer decomposition. 
\end{remark}
\begin{lemma}
We have the following commuting subalgebras living inside $\mathfrak{sp}(2m)$:
	\[ \mathfrak{sl}_{d}(2)\oplus \so(m) \subset \sym(2m).\]
\end{lemma}
\begin{proof}
The infinitesimal rotations, realising the regular action of $\mathsf{SO}(m)$ on the level of the Lie algebra, are defined in terms of the angular momentum operators. On $\mC$-valued functions $f(\ux,\uy;\uz)$, these operators are defined as $$\mathcal{L}_{ab} := L_{ab}^{(x)} + L_{ab}^{(y)} + L_{ab}^{(z)},$$ where for instance $L_{ab}^{(x)}=x_a\partial_{x_b}-x_b\partial_{x_a}$ (and similarly for the other vector variables). It then easily follows that operators $\mathcal{L}_{ab}$ commute with the generators of $\mathfrak{sl}_{d}(2)$.
\end{proof}
\noindent
Putting everything together, we will thus be acting with two commuting realisation of the Lie algebra $\mathfrak{sl}(2)$ on the space $\mcP(\mR^{3m},\mC)$. This is summarised in the following: 
\begin{lemma}\label{commuting} 
\mbox{}
	\begin{enumerate}[\normalfont(i)]
		\item The generators of $\mathfrak{sl}_d(2)$ commute with the symplectic Dirac operator and its dual:
		\[ [R,D_s] = [L,D_s] = 0 = [R,D_s^\dagger] = [L,D_s^\dagger]\ . \]
		\item The following realisation of $\so(4)$ acts on $\mS^\infty_0$-valued functions $f(\ux,\uy;\ux)$ on $\mathbb{R}^{2m}$:
		\[ \so(4) \cong \spl_c(2) \oplus \mathfrak{sl}_d(2) \cong \textup{Alg}\big(D_s,D_s^\dagger\big) \oplus \textup{Alg}(L,R). \]
	\end{enumerate}
\end{lemma}
\noindent
In view of the fact that $\spl_d(2)$ contains differential operators which are 0-homogeneous in the variables $(\ux,\uy)$, in the sense that they commute with $\mE_x + \mE_y$, one can use the operators $R$ and $L$ to find a Fischer decomposition generalising Theorem \ref{decomp}. This is the topic of the theorem below.
\begin{theorem}[$L$-Fischer decomposition]\label{fischerB}
\mbox{}
	\begin{enumerate}[\normalfont(i)]
		\item The space of $\mS^\infty_0$-valued polynomials which are $k$-homogeneous in $(\ux,\uy)$ decomposes as follows: 
		\[ \mcP_k(\mR^{2m},\mS^\infty_0) = \bigoplus_{j = 0}^\infty R^j \ker_k(L)\ . \]
		\item We have the following decomposition for the full space of polynomials on $\mR^{3m}$: 
		\[ \mcP(\mR^{3m},\mC) \cong \bigoplus_{k = 0}^\infty \bigoplus_{j = 0}^\infty R^j \ker_k(L). \]
	\end{enumerate}
\end{theorem}
\begin{proof}
	(i) This follows from the fact that $R$ and $L$ generate the algebra $\spl_d(2)$. Given any (non-trivial) polynomial $P_k(\ux,\uy;\uz) \in \mcP_k(\mR^{2m},\mS^\infty_0)$, with $k$ the total degree in $(\ux,\uy)$, there always exists a unique integer $p \in \mN$ such that $L^{p+1} P_k(\ux,\uy;\uz) = 0$ but $L^p P_k(\ux,\uy;\uz) \neq 0$. Indeed, the operator $L$ essentially consists of two pieces which will eventually act trivially on a polynomial in $(\ux,\uy;\uz)$: either because the action of $\Delta_z$ has cancelled the variable $\uz$, or because the action of $\langle \ux,\up_y \rangle$ has turned all the variables $y_j$ into the correponding variables $x_j$. One can thus consider the polynomial 
	\[ Q_k(\ux,\uy;\uz) := P_k(\ux,\uy;\uz) - c_p R^p\big(L^p P_k(\ux,\uy;\uz)\big), \]
	whereby $c_p \in \mR_0$ is chosen in such a way that $L^p Q_k(\ux,\uy;\uz) = 0$, which means that one can now repeat the argument for this new polynomial (this constant easily follows from the commutation relations in the Lie algebra generated by $L$ and $R$). It is hereby crucial to note that the operators $R$ and $L$ are homogeneous of degree zero in $(\ux,\uy)$, so that we indeed remain in the space of $k$-homogeneous polynomials.\\ 
(ii) This follows from the observation that $\mcP(\mR^{3m},\mC) \cong \mS^\infty_{0^+} \oplus \mS^\infty_{0^-}$, and that the space of polynomials in $(\ux,\uy)$ can be graded with respect to the parameter $k \in \mN$.
\end{proof}

\begin{remark}
	 Due to the property that the operators $D_s$ and $D^\dagger_s$ commute with $R$ and $L$, it is sufficient to decompose $\ker(L)$ into $\ker(D_s)$ and its orthogonal complement. The action of the raising operator $R$ will not `disturb' this (symplectic) decomposition. As a matter of fact, this $R$-action will conveniently explain the (countably many) multiplicities of orthogonal irreducible representations. Put differently: the Verma modules appearing in the rest of the paper will be defined in terms of the operators $R$ and $L$. This generalises the fact that $|\uz|^{2p}$ can be used to label the infinitely many copies of spaces $(a,0,\ldots,0)_z$ appearing in the decomposition of $\mS^\infty_0$. 
\end{remark}
\noindent
In view of the fact that for $\alpha \neq \beta$ we have that
\[ L\big(K_{k,\alpha}(\ux,\uy;\uz) + K_{k,\beta}(\ux,\uy;\uz)\big) = 0\ \Rightarrow\ L K_{k,\alpha}(\ux,\uy;\uz) = 0 = L K_{k,\beta}(\ux,\uy;\uz)\ , \] 
where $\alpha$ and $\beta$ denote eigenvalues for the operator $\mathcal{E}$, we can decompose the space $\ker_k(L)$ of $k$-homogeneous polynomial solutions for $L$ into $\mcE$-eigenspaces. This is completely similar to the situation where the space of harmonic polynomials is decomposed into homogeneous polynomials, using the classical Euler operator. 
\begin{prop}
Let $k \in \mN$ be a fixed degree of homogeneity in $(\ux,\uy) \in \mR^{2m}$.
\begin{enumerate}[\normalfont(i)]
	\item We have the decomposition
	\begin{align}\label{kerkL}
	\ker_k(L) = \bigoplus_\alpha \ker_{k,\alpha}(L)
	\end{align}
	where for all $K_{k,\alpha}(\uz,\uy;\uz) \in \ker_{k,\alpha}(L)$ we have that
	\[ \begin{cases}
	(\mE_x + \mE_y) K_{k,\alpha}(\ux,\uy;\uz)  =k K_{k,\alpha}(\ux,\uy;\uz)\\ \\
	\mcE K_{k,\alpha}(\ux,\uy;\uz) = \alpha K_{k,\alpha}(\ux,\uy;\uz)\ .
	\end{cases} \]
	\item  The possible eigenvalues for $\mcE$ (for fixed $k \in \mN$) are given by the set
	\[ \sigma_k(\mcE) = \left\{\frac{m}{2} - k + j \mid j \in \mN\right\} . \]
	\item For fixed $k \in \mN$ and fixed degree of homogeneity in $\uz$, say $a \in \mN$, the eigenvalues range from $\frac{m}{2} - k + a$ to $\frac{m}{2} + k + a$ (all equal to each other modulo 2).
\end{enumerate}
\end{prop} 
\begin{proof}
It suffices to recall that $\mcE = \mE_y - \mE_x + \mE_z + \frac{m}{2}$ to understand the structure of the spectrum. For fixed degree in $\uz \in \mR^m$, the minimal (resp. maximal) eigenvalue is obtained for a polynomial depending on $\ux$ and $\uz$ only (resp. $\uy$ and $\uz$ only). 
\end{proof}
\noindent
Invoking Theorem \ref{fischerB} (the $L$-Fischer decomposition) and using equality (\ref{kerkL}) we can then decompose $\mcP_1(\mR^{2m},\mS^\infty_0)$ as follows:
\[ \mcP_1(\mR^{2m},\mS^\infty_0) = \bigoplus_{j = 0}^\infty \bigoplus_{a=0}^{\infty} R^j \ker_{1,\frac{m}{2} - 1 + a}(L)\ . \]
The index $\alpha = \frac{m}{2} - 1 + a$ hereby denotes the $\mcE$-eigenvalue. To localise the solutions for $L$ in this space, we will rely on techniques from representation theory (for tensor products). To that end, we note that the operator $L$ can (in tensor notation) also be written as 
 $$L = Y_s \otimes \textup{Id} + \textup{Id} \otimes Y_h,$$ with $Y_s$ and $Y_h$ the `lowering' operator in their respective copies for $\spl(2)$. This means that we can characterise $L$-solutions as lowest weight vectors in tensor products of irreducible representations of $\spl_s(2)$ and $\spl_h(2)$ respectively. Note that these representations are of a completely different nature though: for $\spl_s(2)$ one will always need {\em finite} representations (as the space $\mcP_k(\mR^{2m},\mC)$ is finite), whereas for $\spl_h(2)$ we need {\em infinite} Verma modules. \\
\\ 
To be more precise, we note that for all indices $1 \leq j \leq m$ one has a copy of the standard representation $\mC^2 \cong \mV_1$ for $\spl_s(2)$. This space is spanned by the highest weight vector $y_j$ (weight: $+1$) and the corresponding lowest weight vector $x_j$ (weight: $-1$). For the algebra $\spl_h(2)$ one has the (lowest weight) Verma modules generated by the repeated action of $|\uz|^2$ on a homogeneous harmonic $H_a(\uz)$ of degree $a$ in $\uz\in\R^m$. Now, by observing that the tensor product decomposes a
\begin{align}\label{Verma_tensor}
\mV_1 \otimes \mV^\infty_{a + \frac{m}{2}} \cong \mV^\infty_{a - 1 + \frac{m}{2}} \oplus \mV^\infty_{a + 1 + \frac{m}{2}}\ ,
\end{align}
we expect two contributions to $\ker(L)$. As a matter of fact, the lowest weight vectors of the Verma modules at the right-hand side contain the desired solutions for the operator $L$. It is crucial to point out here that the lowest weight vector belongs to $\ker(L)$, and that the module itself is generated by repeated $R$-action. Put differently, the Verma modules at the right-hand side are irreducible representations for the algebra $\spl_d(2)$. Note also that the tensor product above appears as many times in our $L$-Fischer decomposition as the dimension of $\mC^m$ and $\mcH_a(\mR^m,\mC)$ respectively (since one can freely choose an index $1 \leq j \leq m$ and a basis for the space of harmonic polynomials in $\uz$). Describing these multiplicities in terms of irreducible representations for the algebra $\so(m)$, will allow us to make the $L$-Fischer decomposition multiplicity-free under a dual action.  \\
\\
\noindent
 explicitly construct these $L$-solutions, thereby this providing a proof for the tensor product (\ref{Verma_tensor}) from above.
\begin{lemma}\label{kerL}
	If $\Delta_z P_k(\ux,\uy;\uz) = 0$, we have that $P_k(\ux,\uy;\uz) \in \ker(L^{k+1})$.
\end{lemma}
\begin{proof}
This easily follows from the fact that the operator $L$ reduces to $\langle \ux,\up_y \rangle$ when acting on harmonics in the variable $\uz$.
\end{proof}
\noindent
This means that the explicit description of the contributions to $\ker(L)$ may require a projection on a subspace. In order to have a notation (and a method) which works for all degrees $k \in \mN$, we introduce the following projection operator (see \cite{Zhelobenko}).
\begin{definition}
The \textit{extremal projector} for the Lie algebra $\mathfrak{sl}(2)=\operatorname{Alg}\{X,Y,H\}$ is the idempotent operator $\Pi$ (meaning that $\Pi^2=\Pi$) given by the expression 
\begin{align}
\Pi : = 1 + \sum_{j=1}^{\infty}\frac{(-1)^j}{j!} \frac{\Gamma(H+2)}{\Gamma(H+2+j)}Y^jX^j\ .
\end{align}
This operator satisfies $X\pi = \pi Y = 0$. Note that this operator is defined on an extension $\mathcal{U}'(\mathfrak{sl}(2))$ of the universal enveloping algebra so that formal series containing the operator $H$ in the denominator are well-defined (in practice we will only need a finite summation). 
\end{definition}
\begin{lemma}
The contributions to $\ker_{1,\alpha}(L)$ can be listed as follows: 
\begin{enumerate}[\normalfont(i)]
\item For all $a \in \mN$ and $1 \leq j \leq m$ one has that $L\big(x_j H_a(\uz)\big) = 0$, which means that 
\[ (1)_x \otimes (a)_z \subset \ker_{1,a - 1 + \frac{m}{2}}(L)\ . \]
\item For all $a \in \mN$ and $1 \leq j \leq m$ one has that $L^2\big(y_j H_a(\uz)\big) = 0$, which means that 
\[ \Pi_L\bigg( (1)_y \otimes (a-2)_z \bigg)\subset \ker_{1,a + 1 + \frac{m}{2}}(L)\ .\]
The operator $\Pi_L$ is hereby defined as the projection operator (truncated after the first term), which means that 
\[ \Pi_L\big(y_iH_a(\uz)\big) = \left(1 + \frac{1}{\mcE - 2}RL\right)\big(y_iH_a(\uz)\big) = \frac{(2a + m)y_iH_a(\uz) + x_i|\uz|^2H_a(\uz)}{2a + m - 2}\ .  \]
\end{enumerate}
\end{lemma}
\begin{proof}
\par (i) For $P_1(\ux;\uz) \in \ker(\Delta_z)$ we immediately have that $LP_1(\ux;\uz) = 0$. This means that $(1)_x\otimes (a)_z\in \ker_{1,\alpha}(L)$ where $\alpha$ is then given by $\alpha=a-1+\frac{m}{2}$. \\
(ii) For the second statement, we observe that for $P_1(\uy;\uz)\in \ker(\Delta_z)$ one has that
\[ LP_1(\uy;\uz) = \left(\langle \ux,\up_y \rangle - \frac{1}{2}\Delta_z\right)P_1(\uy;\uz) \neq 0\ .\] 
Invoking {Lemma \ref{kerL}}, it is clear that we can now use the extremal projector $\Pi_L$ to project $P_1(\uy;\uz)$ on the kernel of $L$. In this specific case, the lemma above tells us that the infinite series reduces to a sum of two terms only (as $L^2$ acts trivially). The truncated extremal projector is thus given by
$\Pi_L = 1 + \frac{1}{\mcE - 2}RL$
and is well-defined as the operator $(\mcE - 2)$ always acts non-trivially (recall that $m \geq 6$ in our case of interest). Consequentially, we obtain 
\[ \Pi_L\bigg( (1)_y \otimes (a)_z \bigg)\ \subset\ \ker_{1,\beta}(L)\ , \] 
where the $\mathcal{E}$-eigenvalue is $\beta=\frac{m}{2}+1+a$. 
\end{proof}
\noindent
If one wants to fix the eigenvalue $\alpha$, this gives the following (note that we identify highest $\so(m)$-weights with a space of harmonic polynomials, so that everything should be read as an equality of polynomial spaces): 
\begin{corollary}
Take $\alpha = \frac{m}{2} - 1 + a \in \sigma_1(\mcE)$ fixed. 
\begin{enumerate}[\normalfont(i)]
	\item For $a \in \mN_0$ an even integer we get that
	\begin{eqnarray*}
		\ker_{1,\frac{m}{2} - 1 + a}(L) &=& \bigg((1)_x \otimes (a)_z\bigg)\ \oplus\ \Pi_L\bigg((1)_y \otimes (a-2)_z\bigg)\ 
	\end{eqnarray*}
	In the degenerate case $a=0$ this reduces to $\ker_{1,\frac{m}{2} - 1}(L) = (1)_x$. 
	\item for $a \in \mN \setminus \{1\}$ an odd integer we get that
	\begin{eqnarray*}
		\ker_{1,\frac{m}{2} - 1 + a}(L) &=& \bigg((1)_x \otimes (a)_z\bigg)\ \oplus\ \Pi_L\bigg((1)_y \otimes (a-2)_z\bigg)\ .
	\end{eqnarray*}
	In the degenerate case $a=1$ this reduces to $\ker_{1,\frac{m}{2}}(L) = (1)_x \otimes (1)_z$. 
\end{enumerate}
\end{corollary}
\noindent
This result allows us to rephrase Theorem \ref{fischerB} in the following way, hence making the Fischer decomposition for $\mS^\infty_0$-valued polynomials of degree $k = 1$ more explicit.
\begin{theorem} \label{branching}
In terms of $\mathfrak{so}(m)$-irreducible representations, we have
\begin{align}\label{decomposition}
\mathcal{P}_1(\R^{2m},\mathbb{S}_0^{\infty}) = \bigoplus_{j=0}^{\infty}R^j \left(\nu_1'\oplus \nu_2' \right)\ \oplus\ \bigoplus_{j=0}^{\infty}\bigoplus_{a=2}^{\infty}R^j\left(	\nu_1\oplus\nu_2 	\right)\ .
\end{align}
The components in the decomposition are given by 
\[ \begin{array}{cclccccl}
\nu_1' & = & (1)_x &&& \nu_1 & = & (1)_x \otimes (a)_z\\
\nu_2' & = & (1)_x \otimes (1)_z &&& \nu_2 & = & \Pi_L\bigg((1)_y \otimes (a-2)_z\bigg) 
\end{array}\ .	\]
The primed contributions are the `degenerate' cases (low values for $a \in \mN$). 
\end{theorem}
\begin{proof} It suffices to note that
	\begin{align*}
	\mcP_1(\mR^{2m},\mS^\infty_0) &= \bigoplus_{j = 0}^\infty \bigoplus_{a \in \mN} R^j \ker_{1,\frac{m}{2} - 1 + a}(L) 
	\\&= \bigoplus_{j=0}^{\infty} R^j\left(\left(\ker_{1,\frac{m}{2}-1}(L) \oplus \ker_{1,\frac{m}{2}}(L) \right)\oplus \bigoplus_{a \geq 2}^{\infty}\ker_{1,\frac{m}{2}-1+a}(L)\right)\ ,
	\end{align*}
and to plug in the conclusion of the previous corollary. 
\end{proof}
\noindent
The following table lists the contributions to $\ker_{1,\alpha}(L)$ in a different way: for a given $a \in \mN$ (left-most column) one can see which $\so(m)$-tensor products contribute to the kernel. For each of these tensor products the $\mcE$-eigenvalue is listed too, together with the Verma module obtained by repeated $R$-action on $\ker(L)$. Note that this Verma module appears as many times as the dimension of the tensor product from the second column, as this is how many times the lowest weight vector for the Verma module is counted. 
\begin{corollary}\label{table_ker}
The space $\ker_{1,\alpha}(L)$ can be characterised as follows: 
\[ \begin{array}{|lc|lc|lc|l|}
\hline
a \in \mathbb{N} && \textup{subspace of }\ker_{1,\alpha}(L) && \mcE-\textup{eigenvalue}\ \alpha && \textup{Verma module}\\
\hline\hline
a = 0 && (1)_x && \frac{m}{2} - 1 && \mV^\infty_{\frac{m}{2} - 1}\\
\hline
a = 2j > 0 && (1)_x \otimes (a)_z && \frac{m}{2} - 1 + 2j && \mV^\infty_{\frac{m}{2} +2j - 1}\\
&& (1)_y \otimes (a-2)_z && \frac{m}{2} - 1 + 2j && \mV^\infty_{\frac{m}{2} +2j - 1}\\
\hline
\hline
a = 1 && (1)_x \otimes (1)_z && \frac{m}{2} && \mV^\infty_{\frac{m}{2}}\\
\hline
a = 2j + 1 > 1 && (1)_x \otimes (a)_z && \frac{m}{2} + 2j && \mV^\infty_{\frac{m}{2} +2j}\\
&& (1)_y \otimes (a-2)_z && \frac{m}{2} + 2j && \mV^\infty_{\frac{m}{2} +2j}\\
\hline
\end{array} \]
\end{corollary}
\section{The branching rule for symplectic monogenics}
\noindent
In this section we come to the main goal of this paper: the branching rule for  
\begin{align*}
\mathbb{S}^{\infty}_1\bigg\downarrow^{\mathfrak{sp}(2m)}_{\mathfrak{so}(m)}.
\end{align*} 
In order to make this branching explicit, we can focus on the second column in the table above and decompose these tensor products so that $\so(m)$-irreducible building blocks appear. Once this is done, we still need to project on $\ker(D_s)$ and this will then lead to our final result. First of all, in order to decompose the tensor products $(1) \otimes (a)$, for all $a \in \mN$, we need the following. 
\begin{lemma}[{Klimyk}, \cite{Kl}] 
	Let $a,b\in\mathbb{N}$ such that $a\geq b$. Then, we have the following decomposition rule for the tensor product of irreducible $\mathfrak{so}(m)$-modules (with $m \geq 6$):
\begin{align}\label{Klymik}
	(a,0,\ldots,0) \otimes (b,0,\ldots,0) \cong \bigoplus_{i=0}^b  \bigoplus_{j=0}^{b-i}(a-i+j, b-i-j,0,\ldots,0)\ .
\end{align}
Recall that we will omit trailing zeroes in what follows. 
\end{lemma}
\noindent
In order to interpret this result in terms of harmonics, we recall Definition \ref{simplicial} for the special case $N=2$. It is important to point out that $\uz \in \R^m$ plays a special role again: the first eigenvalue in the notation for a highest weight will always refer to the degree of homogeneity in $\uz$.  
\begin{definition}
	For all $k\geq \ell$, we define the space of so-called \textit{simplicial harmonics} $\mathcal{H}_{k,\ell}$ by means of
		\begin{align}
		\mathcal{H}_{k,\ell}(\mathbb{R}^{2m},\mathbb{C}) = \mathcal{P}_{k,\ell} (\mathbb{R}^{2m},\mathbb{C}) \cap \ker(\Delta_z,\Delta_x,\langle \up_z,\up_x\rangle, \langle \uz,\up_x\rangle). 
		\end{align}
		It defines a model for the irreducible $\mathsf{SO}(m)$-module with highest weight $(k,\ell)$ for $k\geq \ell$ (dominant weight condition). 
\end{definition}
\noindent
In the language of harmonics, formula (\ref{Klymik}) states that the product of two harmonic polynomials can be decomposed in terms of simplicial harmonics. For $a \geq 2$, the tensor products $(1)_{x}\otimes (a)_{z}$ and $(1)_{y}\otimes (a-2)_{z}$ appearing in the second column of table \ref{table_ker} give rise to the following `Klymik triangles' (where trailing zeroes were omitted from highest weights):
\[
\begin{tikzcd}[column sep=1.5em]
(a+1) \arrow[dash]{rr}{} \arrow[swap,dash]{dr}{}& &(a-1) \arrow[dash]{dl}{}\\
& (a,1) & 
\end{tikzcd}
\oplus 
\begin{tikzcd}[column sep=1.5em]
(a-1) \arrow[dash]{rr}{} \arrow[swap,dash]{dr}{}& &(a-3) \arrow[dash]{dl}{}\\
& (a-2,1) & 
\end{tikzcd}
\]
This is a graphical representation, with for instance $(a) \otimes (1) \cong (a,1) \oplus (a+1) \oplus (a-1)$. These decompositions are not completely `explicit' yet, they only hold up to isomorphism. In order to turn them into equalities, we must introduce the appropriate embedding factors (an easy way to see why this is needed is the following: the degrees of homogeneity must match). A systematic way to construct these embedding operators goes back to the work on transvector algebras by Zhelobenko in \cite{Zhelobenko}. To defined this algebra, one starts from a Lie algebra $\mathfrak{g}$ with a subalgebra $\mathfrak{s}\subset\mathfrak{g}$ which is reductive in $\mathfrak{g}$. This means that we can write $\mathfrak{g}= \mathfrak{s}\oplus \mathfrak{t}$, for some subset $\mathfrak{t}\subset\mathfrak{g}$ which is closed under the $\mathfrak{s}$-action defined by the commutator. By fixing a Cartan subalgebra $\mathfrak{h}\subset\mathfrak{s}$, we obtain the triangular decomposition $\mathfrak{s} = \mathfrak{s}^-\oplus\mathfrak{h}\oplus \mathfrak{s}^+$, where the subalgebras $\mathfrak{s}^{\pm}$ contain the positive (resp. negative) roots. If we now consider the left ideal $J:=\mathcal{U}(\mathfrak{g})\mathfrak{s}^+$ generated by $\mathfrak{s}^+$ in the universal enveloping algebra $\mathcal{U}(\mathfrak{g})$, we can construct the \textit{normaliser} $\text{Norm}(J)=\{u\in\mathcal{U}(\mathfrak{g})\mid Ju\subset J\}$. This is a subalgebra $\text{Norm}(J)\subset \mathcal{U}(\mathfrak{g})$ and $J$ is a two-sided ideal of $\text{Norm}(J)$.
The quotient algebra	$\mathcal{S}(\mathfrak{g},\mathfrak{s})=\text{Norm}(J)/J$ is known as the \textit{Mickelsson algebra} in the literature. We will now consider an extension of this algebra by enlarging the universal enveloping algebra $\mathcal{U}(\mathfrak{g})$ in the following way. Denote the field of fractions of the universal enveloping algebra of the Cartan algebra by $\text{Frac}(\mathcal{U}(\mathfrak{h})):=R(\mathfrak{h})$. This allows us to the define 
$ \mathcal{U}'(\mathfrak{g})=\mathcal{U}(\mathfrak{g})\otimes_{\mathcal{U}(\mathfrak{h})}R(\mathfrak{h})$, 
which induces the following two-sided ideal $J':=\mathcal{U}'(\mathfrak{g})\mathfrak{s}^+\trianglelefteq \text{Norm}(J')$. This construction leads to the following:
\begin{definition}
The quotient algebra $\mathcal{Z}(\mathfrak{g},\mathfrak{s})=\text{Norm}(J')/J'$ is called the {Mickelsson-Zhelobenko} algebra, $Z$-algebra or \textit{transvector algebra}. Note that the construction of this algebra makes it possible to divide by the elements of the Cartan algebra. 
\end{definition}
\noindent
To arrive at an explicit construction of the embedding operators turning the isomorphism of the Klimyk decomposition into an equality, we start from the Lie algebra $\mathfrak{g}=\mathfrak{sp}(4)$, hereby following the approach from \cite{BER}. This Lie algebra can be realised as the algebra generated by rotationally invariant operators in two vector variables in $\mR^m$. Further on, this will be needed for both $(\uz,\ux)$ and $(\uz,\uy)$, but to illustrate the construction we will work with the latter. For both variables one can write down a copy of the algebra $\spl_h(2)$, which already accounts for 6 of the 10 generators in $\sym(4)$. Let us denote these Lie algebras by means of $\mathfrak{sl}^z_h(2)$ and $\mathfrak{sl}^y_h(2)$ respectively, where the meaning of the superscript should be clear. One then has that $\mathfrak{sp}(4) = \mathfrak{s} \oplus \mathfrak{t} = \left(	\mathfrak{sl}_{h}^z(2) \oplus \mathfrak{sl}_{h}^y (2) \right) \oplus \mathfrak{t} \cong \mathfrak{so}(4)\oplus \mathfrak{t}$ where $\mathfrak{t}$ is the subspace of $\mathfrak{sp}(4)$ defined by 
$
\mathfrak{t}:=\operatorname{span}\left(\langle \uy,\uz\rangle , \langle \up_y,\up_z \rangle ,\langle \uy,\up_z\rangle, \langle \uz,\up_y\rangle	\right)$.
It is then easily seen that $\mathfrak{t}$ indeed carries an action of $\mathfrak{so}(4)$ under the commutation bracket. Applying the general construction for transvector algebras to this specific case results in the following: 
\begin{definition}\label{Z-gen}
	The generators of the transvector algebra  $\mathcal{Z}(\sym(4),\so(4))$ are given by 
	\begin{align*}
	S_{xz} & = \langle \ux,\up_z \rangle - \frac{|\ux|^2 \langle \up_x,\up_z \rangle}{2\mE_x + m - 4}\\
	S_{zx} & = \langle \uz,\up_x \rangle - \frac{|\uz|^2 \langle \up_x,\up_z \rangle}{2\mE_z + m - 4}\\
	A_{xz} & =  \langle \up_x,\up_z \rangle\\
	C_{xz} & =  \langle \ux,\uz \rangle - \frac{|\ux|^2 \langle \uz,\up_x \rangle}{2\mE_x + m - 4} - \frac{|\uz|^2 \langle \ux,\up_z \rangle}{2\mE_z + m - 4} + \frac{|\ux|^2|\uz|^2 \langle \up_x,\up_z \rangle}{(2\mE_x + m - 4)(2\mE_z + m - 4)}\ ,
	\end{align*}
	whereby any fraction $A/B$ should be read as $B^{-1}A$. The operators $A$ (for `annihilation') and $C$ (for `creation') are symmetric in their subscripts, but for the $S$-operators (for `skew') the order matters. 
\end{definition} 
\noindent
These operators do not form a Lie algebra, as some of the commutation relations are quadratic. We refer to \cite{BER} for an overview of the explicit commutation rules (they will not be needed in the paper).  
We can now use the generators of the algebra $\mathcal{Z}(\sym(4),\so(4))$ to make the isomorphism from Klymik's formula explicit. This is based on the observation that the generators of $\mathcal{Z}(\sym(4),\so(4))$ are all $\mcE$-homogeneous: for instance $S_{xz}$ satisfies $[\mcE,S_{xz}] = -2S_{xz}$, and similarly for the other operators. The following theorem is a refinement of Theorem \ref{branching} which essentially tells us how to rewrite each tensor product in the second column of our Table \ref{table_ker}. 
\begin{theorem}\label{contributions}
	By fixing an eigenvalue $\alpha := \frac{m}{2} - 1 + a$ with $a \in \mN$, we obtain the following contributions to $\ker_{1,\alpha}(L)$ in terms of $\so(m)$-irreducible representations: 
	\begin{enumerate}[\normalfont(i)]
		\item From the polynomials which have degree $1$ in $\ux \in \mR^m$, we get: 
		\[ \mcH_{a,1}(\mR^{2m},\mC) \oplus S_{xz}\mcH_{a + 1}(\mR^m,\mC) \oplus C_{xz}\mcH_{a-1}(\mR^m,\mC).\]
		\item From the polynomials which have degree $1$ in $\uy \in \mR^m$, we get:  
		\[ \Pi_L\bigg(\mcH_{a-2,1}(\mR^{2m},\mC) \oplus S_{yz}\mcH_{a - 1}(\mR^m,\mC) \oplus C_{yz}\mcH_{a-3}(\mR^m,\mC)\bigg). \]
	\end{enumerate}
Note that whenever the highest weight (the subscript attached to the spaces of polynomials) is non-dominant, it can be ignored. 
\end{theorem}
\noindent
To arrive at the branching rule, we supplement the information from the theorem above with the Fischer decomposition for the space $\mcP(\mR^{2m},\mS^\infty_0)$ (see Theorem \ref{decomposition}). For $k=1$ we have that $\mcP_1(\mR^{2m},\mS^\infty_0) \cong \mS^\infty_1 \oplus D^\dagger_s \mS^\infty_0$. We already know how $\mS^\infty_0$ behaves under the branching (Theorem \ref{s0infty}), so in a sense we only need to `subtract' the Verma modules appearing in $\mS^\infty_0$ from the decomposition obtained for $\mcP_1(\mR^{2m},\mS_0)$. This `subtraction' is rather intricate though, in view of the countable infinities on the level of multiplicities. However, using the explicit realisation from above (in terms of transvector generators) one can look at the interplay between these generators and the symplectic Dirac operator $D_s$. Put differently: one can let $D_s$ act on the six summands from the theorem above, and all that remains is {\em not} part of $\mS^\infty_1$. \\
\noindent
Let us first of all give an {\em intuitive} argument, showing that one can localise all candidates in table \ref{table_ker} which should appear as part of $\mS^\infty_0$ (in a sense the complement of the spaces we are interested in). Note that these candidates are listed as irreducible representations for the action of SO$(m) \times \spl_h(2)$. 
\begin{enumerate}[(i)]
	\item  If we consider the \textit{even} degrees of homogeneity in $\uz \in \mR^m$ first, we need to localise $(2j) \otimes \mV^\infty_{\frac{m}{2} + 2j}$ in the table, whereby $(2j)$ refers to $\mcH_{2j}(\mR^m,\mC)$. 
\begin{enumerate}
\item From the third row (the case $a = 1$), we will get a contribution of the form 
\[ C_{xz}(0,\ldots,0) \subset (1)_x \otimes (1)_z. \] 
Under the repeated $R$-action, this should gives rise to a Verma module labelled by its lowest weight $\frac{m}{2}$. One can easily verify that the action of $C_{xz}$ on $\mC$ reduces to multiplication with $\langle \ux,\uz \rangle$, see definition \ref{Z-gen}, which is exactly what we would find for the action of $D_s^\dagger$ too.
\item From the fourth row (for odd $a > 1$), we get contributions of the form 
		\[ \left( \begin{array}{ccc}
		(a + 1) && \boxed{(a - 1)}\\
		&(a,1)&
		\end{array} \oplus  
		\begin{array}{ccc}
		\boxed{(a - 1)} && (a - 3)\\
		&(a-2,1)&
		\end{array}\right), \]
and these appear with a lowest weight Verma module labelled as $\mV^\infty_{\frac{m}{2} + (a-1)}$. This means that we expect the boxed summands to split into a part which belongs to $\ker(D_s)$ and a part which belongs to the orthogonal complement (contributing to the space $D_s^\dagger \mS^\infty_0$). This will be confirmed in the theorem below. 
		\item All other contributions coming from the lower half of the table are thus expected to contribute to $\mS^\infty_1$.  
	\end{enumerate}
	\item A similar thing can be done for \textit{odd} degrees of homogeneity in $\uz$ as well: this time we would like to use table \ref{table_ker} to localise tensor products of the form $(2j+1) \otimes \mV^\infty_{\frac{m}{2} + 2j + 1}$ for all $j \in \mN$. 
	\begin{enumerate}
		\item From the first row (for $a = 0$), we will get a contribution to $\mS^\infty_1$ (the corresponding Verma module does not appear in the list of Verma modules for $\mS^\infty_0$). 
		\item From the second row (for even $a > 0$), we get all the contributions of the form 
		\[ \left( \begin{array}{ccc}
		(a + 1) && \boxed{(a - 1)}\\
		&(a,1)&
		\end{array} \oplus  
		\begin{array}{ccc}
		\boxed{(a - 1)} && (a - 3)\\
		&(a-2,1)&
		\end{array}\right), \]
accompanied by the Verma module $\mV^\infty_{\frac{m}{2} + (a-1)}$ for $a = 2j$ an even number. We thus again expect the boxed summands to split into a part which belongs to $\ker(D_s)$ and a part which belongs to the orthogonal complement (contributing to the space $D_s^\dagger \mS^\infty_0$).
	\end{enumerate}
\end{enumerate}
All in all, we now know where to look for $D_s^\dagger \mS^\infty_0$, and this means that all other components should contribute to the summands appearing in the branching of $\mS^\infty_1$. Note that these `components' are of the form $(p,q) \otimes \mV_{\alpha}$, with $(p,q)$ a highest weight for $\so(m)$ and $\mV_\alpha$ a lowest weight Verma module generated by repeated $R$-action. The former are the desired branched components, the latter tells us how to organise the infinite multiplicities. Let us then have a closer look at the calculations which can support the abstract arguments. 

\begin{theorem}
The following highest $\mathfrak{so}(m)$-weights appear in the $1$-homogeneous kernel space of the symplectic Dirac operator: 
\[ (a,1), (a + 1), (a-2,1), (a-3), (a-1). \]
Each weight is defined for all values of $a \in \mN$ for which it exists (taking the dominant weight condition into account). Each of these representations appears with infinite multiplicity and can be interpreted as a Verma module (see the table after this theorem). 
\end{theorem}
\begin{proof}
We will essentially go through the building blocks for $\ker_{1,\alpha}(L)$, see once again Table \ref{table_ker} for the list: 
\begin{enumerate}[\normalfont(i)]
\item Using the fact that $H_{a,1}(\uz;\ux)$ is independent of $\uy$ and that it is a simplicial harmonic, which means that $\langle\underline{\partial}_x,\underline{\partial}_z \rangle H_{a,1}(\uz;\ux)$ we immediately obtain that $D_s H_{a,1}(\uz;\ux) = 0$. 
\item One directly observes that $D_s \langle \ux,\up_z \rangle H_{a+1}(\uz) = -\Delta_z H_{a+1}(\uz) = 0$. 
\item To show that $D_s \Pi_L H_{a-2,1}(\uz;\uy) = 0$, we make use of the fact that the projection operator $\Pi_L$ commutes with $D_s$ (it is defined in terms of $R$ and $L$), and simply note that $D_s H_{a-2,1}(\uz;\uy) = 0$ by definition of the simplicial harmonics. 
\item We have that $D_s \Pi_L C_{yz}H_{a-3}(\uz) = \Pi_L D_s C_{yz} H_{a-3}(\uz)$. Plugging in the explicit definition for $C_{yz}$, for which we refer to definition \ref{Z-gen}, it is clear that
\[  D_sC_{uz}H_{a-3}(\uz) = D_s\left( \langle \uy,\uz \rangle - \frac{|\uz|^2 \langle \uy,\up_z \rangle}{2a + m - 4}\right)H_{a-3}(\uz) = c_{a,m}|\uz|^2 H_{a-3}(\uz),  \]
with $c_{a,m}$ a numerical (albeit irrelevant) constant. 	This is clearly not trivial, but we still have to let the operator $\Pi_L$ act, and on functions depending on $\uz \in \mR^m$ only this projection operator reduces to a projection on harmonics in $\uz$. Due to the presence of the factor $|\uz|^2$, we indeed get zero as desired. 
\item Finally, let us focus on the two copies of $(a-1)$ appearing in the `Klymik triangles'. The first component appears as $C_{xz}\mathcal{H}_{a-1}(\mR^{m},\mC)$, which can be simplified as 
\begin{align*}
\left( \langle \ux,\uz \rangle - \frac{|\uz|^2 \langle \ux,\up_z \rangle}{2\mE_z + m - 4}\right)H_{a - 1}(\uz)
	= \left( \langle \ux,\uz \rangle - \frac{|\uz|^2 \langle \ux,\up_z \rangle}{2a + m - 4}\right)H_{a - 1}(\uz)
\end{align*}
and the second as $\Pi_L S_{xz}\mathcal{H}_{a-1}(\mR^{m},\mC)$. Using the explicit form for the (truncated) projection operator, this can be rewritten as 
\begin{align*}
\left(\langle \uy,\up_z \rangle + \frac{R \langle \ux,\up_z \rangle}{a + \frac{m}{2} - 3}\right)H_{a - 1}(\uz)= \left(1 + \frac{R L}{a + \frac{m}{2} - 3}\right)\langle \uy,\up_z \rangle H_{a - 1}(\uz).
\end{align*} 
This can be slightly rewritten, in the sense that 
	\[ (2a + m - 4)C_{xz}H_{a - 1}(\uz) = \big( (2a + m - 4)\langle \ux,\uz \rangle - |\uz|^2 \langle \ux,\up_z \rangle \big)H_{a - 1}(\uz) \]
	and also 
	\[ (2a + m - 6)\Pi_L \langle \uy,\up_z \rangle H_{a - 1}(\uz) = \big( (2a + m - 6)\langle \uy,\up_z \rangle + |\uz|^2 \langle \ux,\up_z \rangle \big)H_{a - 1}(\uz). \]
	From this we can clearly see that there will be a contribution of the form $D_s^\dagger H_\ell(\uz)$, such that the complement contributes to $\ker(D_s)$.
\end{enumerate}
This exhausts all summands, and therefore finishes the proof. 
\end{proof}
\noindent
Putting everything together, we can now list all the contributions to the branching problem for $\mS^\infty_1$. This is given as a list of all the Verma modules (keeping track of the countable multiplicities) together with the list of highest $\so(m)$-weights (ignoring all trailing zeroes). We also added the possible values for $a \in \mN$ as a range condition. 
$$\begin{tabular}{|c|c|c|}
\hline 	 
{Verma module} & $\so(m)$-weight & range \\ 
\hline \hline
$\mathbb{V}^{\infty}_{\frac{m}{2}+a-2}$	& $(a)$ & $a \geq 1$  \\ 
\hline 
$\mathbb{V}^{\infty}_{\frac{m}{2}+a-1}$	&  $(a,1)$ & $a \geq 1$  \\ 
\hline 
$\mathbb{V}^{\infty}_{\frac{m}{2}+a}$  &  $(a)$ & $a \geq 1$ \\ 
\hline 
$\mathbb{V}^{\infty}_{\frac{m}{2}+a + 1}$  &  $(a,1)$ & $a \geq 1$ \\
\hline
$\mathbb{V}^{\infty}_{\frac{m}{2}+a + 2}$  &  $(a)$ & $a \geq 0$ \\
\hline
\end{tabular} 
$$
\begin{theorem}[Branching rule]
The branching of the irreducible $\sym(2m)$-representation ${\mathbb{S}}_1^{\infty}$ with respect to the subalgebra $\so(m)$ can be described in a multiplicity-free way as a direct sum of irreducible $\so(m) \times \spl_d(2)$-representations, given by: 
$$\begin{tabular}{|l|c|}
\hline 	 
$(a) \otimes \mathbb{V}^{\infty}_{\frac{m}{2}+a-2}$ & $a \geq 1$  \\ 
\hline 
$(a,1) \otimes \mathbb{V}^{\infty}_{\frac{m}{2}+a-1}$	&  $a \geq 1$  \\ 
\hline 
$(a) \otimes \mathbb{V}^{\infty}_{\frac{m}{2}+a}$  &  $a \geq 1$ \\ 
\hline 
$(a,1) \otimes \mathbb{V}^{\infty}_{\frac{m}{2}+a + 1}$  &  $a \geq 1$ \\
\hline
$(a) \otimes \mathbb{V}^{\infty}_{\frac{m}{2}+a + 2}$  &  $a \geq 0$ \\
\hline
\end{tabular} 
$$
\end{theorem}
\noindent
From this, one can extract a list of Verma modules for $\spl(2)$ which appear when branching the oscillator representation. This may seem odd, as the branching should give rise to a list of $\so(m)$-weights, but these are then `hidden' in the multiplicity with which this Verma module appears.

\begin{enumerate}[\normalfont (i)]
\item $\mathbb{V}^{\infty}_{\frac{m}{2}-1}$ appears with multiplicity $\dim(\mathcal{H}_1)$;
\item $\mathbb{V}^{\infty}_{\frac{m}{2}}$ appears with multiplicity $\dim(\mathcal{H}_{1,1})+\dim(\mathcal{H}_{2})$;
\item  $\mathbb{V}^{\infty}_{\frac{m}{2}+1}$ appears with multiplicity $\dim(\mathcal{H}_3)+ \dim(\mathcal{H}_{2,1})+\dim(\mathcal{H}_1)$;
\item $\mathbb{V}^{\infty}_{\frac{m}{2}+1+a}$ for $a > 1$ appears with multiplicity given by
$$\dim(\mathcal{H}_{a+3})+\dim(\mathcal{H}_{a+2,1})+\dim(\mathcal{H}_{a+1})+\dim(\mathcal{H}_{a,1})+\dim(\mathcal{H}_{a-1})\ .$$
\end{enumerate} 
Note that it is {\em not} sufficient to work with (iv) only, hereby adding that one should solely take the dominant weights into account, as this would give the wrong conclusion in (ii). \\
\\
\noindent
This list can be obtained in a different way, using another approach which emphasises the role played by the dual symmetry algebra $\spl(2)$. The advantage of this approach is that it generalises to arbitrary $k$-values (with $k$ the degree of homogeneity of a symplectic spinor), the disadvantage is that one loses track of the explicit embedding factors which thus somehow obscures the conclusion. To illustrate what we mean by this, let us once again start from the harmonic Fischer decomposition:
\[ \mS^\infty_0 \cong \mathcal{P}(\R^{m},\C)  = \bigoplus_{a=0}^{\infty}\mathcal{H}_a(\R^m,\C)\otimes \mathbb{V}_{a+\frac{m}{2}}^{\infty}\ . \]
Similarly, as a representation for $\spl_s(2)$ one can say that $\mathcal{P}_1(\mathbb{R}^{2m},\mathbb{C}) \cong \mcH_1(\mR^m,\mC) \otimes \mV_1$, whereby $\mcH_1$ is added to keep track of the multiplicity with which $\mV_1$ appears. Identifying this space $\mcH_1$ with the span of the variables $x_i$ (for $1 \leq i \leq m$), the space $\mV_1$ can then be seen as span$_\mathbb{C}\{x_i,y_i\}$ for each of these variables. Putting both tensor products together, we thus have the following decomposition into $\so(m) \times \spl(2)$-representations:
\begin{align*}
\mathcal{P}_1(\mathbb{R}^{2m},\mathbb{C})\otimes \mathcal{P}(\R^m,\C) &\cong \bigoplus_{a \in \mN}\left(\mathcal{H}_1\otimes \mathbb{V}_1\right) \otimes \left(\mathcal{H}_a \otimes \mathbb{V}^{\infty}_{\frac{m}{2}+a}	\right) \\
&\cong \bigoplus_{a \in \mN}\left( \mathcal{H}_1 \otimes \mathcal{H}_a	\right) \otimes \left( \mathbb{V}_1 \otimes \mathbb{V}^{\infty}_{\frac{m}{2}+a} \right)\ ,
\end{align*}
where we have grouped the representations for $\so(m)$ and $\spl(2)$ respectively. Decomposing the tensor products for $\spl(2)$ gives: 
\begin{align*}
\mathcal{P}_1(\mathbb{R}^{2m},\mathbb{C})\otimes \mathbb{S}_0^{\infty} 
&\cong \bigoplus_{a \in \mN} \bigg((1) \otimes (a)\bigg)\otimes \left( \mathbb{V}^{\infty}_{\frac{m}{2}+a+1}\oplus \mathbb{V}^{\infty}_{\frac{m}{2}+a-1}	\right)\ .
\end{align*}
Also the tensor products for $\so(m)$ can be calculated (using Klymik's rule), but then one needs to take into account that only dominant weights appear (for $a \in \{0,1\}$ certain $\so(m)$-weights should thus be omitted). Although this is in a sense less transparent, there is no information regarding embedding factors, it does mean that one can again make a list of the $\spl(2)$-modules appearing inside the space $\mcP(\mR^{2m},\mS^\infty_0)$ together with their multiplicities (once again expressed in terms of dimensions of $\so(m)$-representations): 
\begin{lemma}\label{S1} 
Under the action of the algebra $\so(m) \times \spl(2)$ on the space $\mathcal{P}_1(\mathbb{R}^{2m},\mathbb{C})\otimes \mathbb{S}_0^{\infty}$, one has that the following irreducible $\spl(2)$-representations will appear: 
	\begin{enumerate}[\normalfont(i)]
		\item $\mathbb{V}^{\infty}_{\frac{m}{2}-1}$ with multiplicity $m = \dim(\mcH_1)$;
		\item $\mathbb{V}^{\infty}_{\frac{m}{2}}$ with multiplicity $m^2 = \dim(\mcH_0 \oplus \mcH_{1,1} \oplus \mcH_2)$;
		\item	$\mathbb{V}^{\infty}_{\frac{m}{2}+a}$ for $a \geq 1$ with multiplicity $m\left(\dim(\mathcal{H}_{a-1}) + \dim(\mathcal{H}_{a+1})\right)$.
	\end{enumerate}
\end{lemma}
\begin{proof}
For this purpose, it suffices to look at the decomposition obtained above. For (i), only $\mcH_1$ will contribute, whereas for (ii) we get a multiplicity of the form $\mcH_1 \otimes \mcH_1$. For the generic case with $a \geq 1$, there is a contribution from $(1) \otimes (a\pm 1)$, leading to the desired dimensions. 
\end{proof}
\begin{theorem}
Under the action of the algebra $\so(m) \times \spl(2)$ on the space ${\mathbb{S}}_1^{\infty}$, the following irreducible $\spl(2)$-representations will appear:
	\begin{enumerate}[\normalfont(i)]
		\item $\mathbb{V}^{\infty}_{\frac{m}{2}-1}$ with multiplicity $m$;
		\item $\mathbb{V}^{\infty}_{\frac{m}{2}}$ with multiplicity $m^2-1$;
		\item	$\mathbb{V}^{\infty}_{\frac{m}{2}+a}$ with multiplicity $m\left(\dim(\mathcal{H}_{a-1}+\dim(\mathcal{H}_{a+1}))\right)-\dim(\mathcal{H}_a)$.
	\end{enumerate}
	
\end{theorem}
\begin{proof}
The symplectic Fischer decomposition says that $\mathcal{P}_1(\mR^{2m},\mC) \otimes\mathbb{S}_0^{\infty} = D_s^{\dagger}\mathbb{S}_0^{\infty} \oplus \mathbb{S}_1^{\infty}$, where the second summand on the right-hand side is the module we are interested in. This means that we need to `subtract' the contributions coming from $\mathbb{S}_0^{\infty}$ to obtain information about $\mS^\infty_1$. Due to the fact that under the action of $\so(m) \times \spl(2)$ we know that 
\[ \mathbb{S}_0^{\infty} = \mathcal{P}(\R^m,\mathbb{C}) \cong \bigoplus_{a=0}^{\infty} \mathcal{H}_a \otimes \mathbb{V}^{\infty}_{\frac{m}{2}+a}\ . \]
This means that subtracting $\dim(\mathcal{H}_a)$ times the Verma module $\mathbb{V}^{\infty}_{\frac{m}{2}+a}$ from the full space leaves us with the $\spl(2)$-Verma modules (together with their multiplicities) which should appear in $\mS^\infty_1$. Together with the information coming from lemma \ref{S1} above, this leads to the desired result.  
\end{proof}
\noindent
We can now see that the same Verma modules appear in both approaches (purely looking at their labels, i.e. the lowest weights). That also the multiplicities agree is not immediately clear, and should be verified explicitly (at least for the generic case, for small values of $a$ it is immediate). This is the content of the following proposition.
\begin{prop}
The following dimensional equality holds (for $a \geq 2$):
	\begin{align*}
	& m\left(\dim(\mathcal{H}_{a+1})+\dim(\mathcal{H}_{a-1})\right) -\dim(\mathcal{H}_a)\\
	&= \dim(\mathcal{H}_{a+1,1}) +\dim(\mathcal{H}_{a-1,1})+\dim(\mathcal{H}_{a+2}) + \dim(\mathcal{H}_{a})+ \dim(\mathcal{H}_{a-2}) 
	\end{align*}
\end{prop}
\begin{proof}
Using the information coming from the Klymik triangle, we have that 
\begin{align*}
\dim(\mcH_{a+1,1}) &= \dim(\mcH_{a+1})\dim(\mcH_1) - \dim(\mcH_{a+2}) - \dim(\mcH_a)\\
\dim(\mcH_{a-1,1}) &= \dim(\mcH_{a-1})\dim(\mcH_1) - \dim(\mcH_{a}) - \dim(\mcH_{a-2})\ ,
\end{align*}
which means that the right-hand side reduces to 
\[ \dim(\mcH_1)\big(\dim(\mcH_{a+1}) + \dim(\mcH_{a-1})\big) - \dim(\mcH_a)\ . \]
This is indeed equal to the left-hand side. 
\end{proof}
\noindent

\section{Conclusion and further research}
In {this paper} we solved the branching problem by introducing the algebra $\mathfrak{sl}_d(2)$ which commutes with $\mathfrak{so}(m)\subset\mathfrak{sp}(2m)$. By decomposing $\ker_1(L)$ into $\mathcal{E}$-eigenspaces we were able to list the components in the branching by introducing suitable embedding operators by means the transvector algebra $\mathcal{Z}(\mathfrak{sp}(4),\so(4))$, which resulted in a branching rule which can be expressed as {\em an equality}. \\
\noindent
At the same time, we also illustrated how this branching rule could be obtained as {\em an isomorphism}, hereby ignoring the embedding factors. In an upcoming paper we will use the latter approach to consider the branching problem for $\mS^\infty_k$ with $k \leq 2$, hereby using an approach which leads to an isomorphism. However, to conclude this paper we will briefly illustrate how this works: 
\begin{enumerate}
\item[(i)] First, one needs to know how the space $\mcP_k(\mR^{2m},\mC)$ of $k$-homogeneous polynomials in $(\ux,\uy)$ behaves as a representation space for the Lie algebra $\spl_s(2)$ in terms of the operators $\langle \ux,\up_y \rangle$ and $\langle \uy,\up_x \rangle$. To do so, we can rely on the fact that there is a Howe dual pair $\sym(4) \times \so(m)$ acting on the full space of polynomials in a matrix variable $(\ux,\uy) \in \mR^{2m}$ (see {Theorem 2.1 in \cite{BER} for the full statement and the references therein for a formal proof}), which allows us to decompose these polynomials into a sum of products of invariants $I(\ux,\uy)$ in the (commutative) algebra $ \mbox{Alg}(|\ux|^2,\langle \ux,\uy \rangle,|\uy|^2)$ times so-called Howe harmonics $G(\ux,\uy) \in \ker(\Delta_x,\Delta_y,\langle \up_x,\up_y \rangle)$. For both pieces, the behaviour under $\spl_s(2)$ can be characterised (these Howe harmonics can be understood as specific combinations of simplicial harmonics), which thus means that one can use the Clebsch-Gordon rules to understand how the space $\mcP_k(\mR^{2m},\mC)$ decomposes into $\spl_s(2)$-spaces.
\item[(ii)] Next, we know that these irreducible representations for $\spl_s(2)$ will appear with multiplicities, but these can be tracked using the behaviour under the action of $\so(m)$. From the invariants, there is no contribution, whereas from the spaces of simplicial harmonics there is a contribution of the form $(k,\ell,0,\ldots,0)$. 
\item[(iii)] When making everything $\mS^\infty_0$-valued, this reduces the problem to understanding tensor products of the form $\mV_p \otimes \mV^\infty_{a + \frac{m}{2}}$ for $\spl(2)$, and of the form $(k,\ell) \otimes (a)$ for $\so(m)$. The latter is slightly more complicated, and requires a generalisation of the Klymik rule which was used in this paper. 
\item[(iv)] Finally, once $\mS^\infty_{j}$ is understood for all $0 \leq j < k$, an inductive argument will allow us to list the spaces which contribute to $\mS^\infty_k$. This will then be a list of the form $\mcH_{p,q,r}(\mR^{3m},\mC) \otimes \mV^\infty_\lambda$, with $\lambda$ a suitable lowest weight for a Verma module for $\spl(2)$. 
\end{enumerate}


\nocite{*}
\bibliographystyle{apa}

\end{document}